\documentclass[12pt]{article}
\usepackage[a4paper, margin=2.3cm]{geometry}
\usepackage{amssymb,amsthm, amsmath}
\usepackage{hyperref}
\usepackage{color}%for colored comments

\newtheorem{theorem}{Theorem}[section]

\newtheorem{lemma}[theorem]{Lemma}

\theoremstyle{definition}

\newcommand{\R}{\mathbb R}
\newcommand{\F}{\mathbb F}

\newcommand{\proj}{\mathbb{P}}

\author{Frank de Zeeuw}
\title{A short proof of Rudnev's point-plane incidence bound}

\begin{document}
\date{}
\maketitle

\begin{abstract}
In this note we give a shortened proof of a theorem of Rudnev, 
which bounds the number of incidences between points and planes over an arbitrary field.
Rudnev's proof uses a map that goes via the four-dimensional Klein quadric to a three-dimensional space, 
where it applies a bound of Guth and Katz on intersection points of lines.
We describe a simple geometric map that directly sends point-plane incidences to line-line intersections in space, 
allowing us to reprove Rudnev's theorem with fewer technicalities.
\end{abstract}

\section{Introduction}

Let $\F$ be any field.
Given a point set $P\subset \F^3$ and a set $Q$ of planes, 
we define the number of \emph{incidences} between $P$ and $Q$ to be $I(P,Q) = |\{(p,q)\in P\times Q: p\in q\}|$.
Rudnev \cite{R} proved the following incidence bound.

\begin{theorem}\label{thm:rudnev}
Let $P$ be a finite set of points in $\F^3$
and $Q$ a finite set of planes in $\F^3$.
Assume that $|P|\leq |Q|$, 
and if $\F$ has positive characteristic $p$,
assume that $|P|=O(p^2)$.
Suppose that no line contains $k$ points of $P$.
%Let $k$ be the maximum number of collinear points in $P$.
Then
\[ I(P,Q) = O\left(|P|^{1/2}|Q| + k|Q|\right).\]
\end{theorem}

This bound is optimal for $k\geq |P|^{1/2}$, 
since then we can take $k-1$ points on a line and all planes of $Q$ containing that line, to get $(k-1)|Q|$ incidences.
On the other hand, for small $k$ there are known to be better bounds over $\R$ (see \cite{R}), 
but it is unknown if such bounds hold over other fields.
Theorem \ref{thm:rudnev} has already seen many applications, including the current best sum-product bound over finite fields \cite{RRS} and the current best point-line incidence bound over finite fields \cite{SZ}.

The point of this note is to give an alternative proof of Theorem \ref{thm:rudnev} that is shorter and simpler than that of Rudnev \cite{R}.
Rudnev's proof maps points and planes in $\F^3$ to planes in the Klein quadric (a four-dimensional variety in $\F^5$ that parametrizes lines in $\F^3$), in such a way that a point is incident to a plane if and only if the corresponding planes intersect in a line.
Intersecting with a generic hyperplane, 
these plane-plane intersections become line-line intersections in a three-dimensional variety, which can be bounded using a customized version of a result of Guth and Katz \cite{GK}.
Some work is required to show that the parameter $k$ in Theorem \ref{thm:rudnev} corresponds to a similar parameter in the Guth--Katz theorem.

Our proof uses a direct map from point-plane incidences to line-line intersections.
The proof is close in spirit to that of Rudnev, and certainly inspired by it.
However, by staying in three dimensions, several technical complications disappear,
and we can use a more straightforward variant of the Guth--Katz bound.
We can in fact strengthen Theorem \ref{thm:rudnev} slightly, by weakening the condition on collinear points (see Theorem \ref{thm:rudnevrefined}).

%%\newpage
\section{Mapping points and planes to lines and lines}\label{sec:map}

We first give a parametrization of the family of lines in $\F^3$ passing through a fixed line $\lambda$ in $\F^3$.
For concreteness, 
we choose $\lambda$ to be the $z$-axis.
The parametrization also uses a plane $\pi$ that is disjoint from $\lambda$ (in affine space);
for concreteness,
we choose $\pi$ to be the plane defined by $x=1$.
Given a line $\ell\subset \F^3$ that intersects $\lambda$ and $\pi$, 
we write its intersection point with $\lambda$ in the form $(0,0,a)$
and its intersection point with $\pi$ in the form $(1,b,c)$,
and we set 
\[\ell^* = (a,b,c)\in \F^3.\]

Given a point $p\in \F^3$, 
we define 
\[\varphi(p) = \{\ell^*: ~ \text{$\ell$ contains $p$ and intersects $\lambda$ and $\pi$} \}.\]
In other words, we consider the pencil of lines through $p$ hitting $\lambda$, and we represent the lines of this pencil in the parametrization above.
For a plane $q\subset \F^3$, 
we define
\[\psi(q) = \{\ell^*: ~ \text{$\ell$ is contained in $q$ and intersects $\lambda$ and $\pi$} \}.\]
In other words, we consider the pencil of lines contained in $q$ that pass through the intersection point of $q$ and $\lambda$, and represent it as above.

The following lemma shows that typically $\varphi(p)$ and $\psi(q)$ are lines, 
which intersect if and only if $p$ is incident to $q$.
After proving the lemma, we forget about the parametrization, 
and we only use the properties stated in this lemma.

\begin{lemma}\label{lem:map}
Let $p\in \F^3$ be a point outside the $yz$-plane,
and let $q$ be a plane that intersects $\lambda$ and $\pi$ but does not contain $\lambda$.
Then $\varphi(p)$ is a line contained in a plane of the form $y=y_0$, and $\psi(q)$ is a line contained in a plane of the form $x=x_0$.
We have an incidence $p\in q$ if and only if the line $\varphi(p)$ intersects the line $\psi(q)$.
\end{lemma}
\begin{proof}
%First we consider $\varphi(p)$.
If $p\in \F^3$ is not on the $yz$-plane, then the plane through $p$ and $\lambda$ intersects $\pi$ in a line of the form $y=y_0$,
so $\varphi(p)$ lies in the plane $y=y_0$.
If a line $\ell$ through $p$ intersects $\lambda$ in the point $(0,0,t)$, 
then $\ell$ intersects $\pi$ in the point $(1,y_0,ut+v)$, for fixed $u,v$.
Then $\varphi(p)$ has the form $\{(t, y_0, u t+v):t\in \F\}$, so it is a line.

If a plane $q$ intersects $\lambda$ in a unique point $(0,0,x_0)$,
and $\pi$ in the line $\{(1,t,u t + v): t\in \F\}$,
then $\psi(q)$ is the line $\{(x_0, t, ut+v): t\in \F\}$.

We have an incidence $p\in q$ if and only if there is a unique line $\ell$ intersecting $\lambda$ with $p\in \ell \subset q$, 
which is equivalent to the point $\ell^*$ lying on both $\varphi(p)$ and $\psi(q)$.
\end{proof}

The fact that each $\varphi(p)$ is contained in a plane of the form $y=y_0$ means that there is a line at infinity such that every $\varphi(p)$ intersects that line.
Similarly, there is a (different) line at infinity such that every $\psi(q)$ intersects that line.

We note here that if the points $p_1,\ldots,p_m$ lie on a line $\ell$ that is contained in the planes $q_1,\ldots,q_n$,
with $m,n\geq 3$, 
then the lines $\varphi(p_i)$ and $\psi(q_j)$ all lie on the same quadric (see Section \ref{sec:guthkatz}).
Indeed, any three lines lie on a quadric, so there is a quadric $S$ containing $\varphi(p_1),\varphi(p_2),\varphi(p_3)$.
Since each $q_j$ contains $p_1,p_2,p_3$, 
the line $\psi(q_j)$ intersects the lines $\varphi(p_1),\varphi(p_2),\varphi(p_3)$,
so $\psi(q_j)$ intersects $S$ in at least three points.
This implies that each $\psi(q_j)$ is contained in $S$.
By a similar argument it then follows that the lines $\varphi(p_4),\ldots, \varphi(p_m)$ lie on $S$.
We will use a related observation in the proof of Theorem \ref{thm:rudnev}.

It would be natural to do the above in projective space $\proj^3$, because then various exceptions would not occur (for instance, every line would intersect $\pi$).
However, the parameter space would then be $\proj^1\times \proj^2$ (and not $\proj^3$), 
and this would introduce other complications.
This is why we have opted for the affine formulation.
But note that the parametrization could be made to work for any choice of line $\lambda$ and plane $\pi$ in $\proj^3$, as long as $\lambda\not\subset\pi$.

\newpage
\section{A bipartite version of the Guth--Katz intersection bound}\label{sec:guthkatz}

We now prove the variant of the Guth--Katz intersection bound that we use. 
It differs from the original intersection bound of Guth and Katz \cite{GK} in that it bounds the number of points where a line from one set intersects a line from another set,
whereas \cite{GK} bounds the number of points where two lines intersect that come from the same set of lines.
In other words, the statement here is a `bipartite' version of that in \cite{GK}.

Another difference is that we prove the bound over arbitrary fields, whereas the bound in \cite{GK} is stated over $\R$.
Koll\'ar \cite{K} showed that 
the Guth--Katz intersection bound also holds over finite fields, if one stipulates that the set of lines is not too large compared to the characteristic\footnote{Note that this is not known to be true for the bound of \cite{GK} on point-line incidences in $\R^3$.}.
Rudnev \cite{R} proved a bipartite version over arbitrary fields, but his statement is customized to the situation in his proof of Theorem \ref{thm:rudnev}.
Here we use the setup and the tools of Koll\'ar \cite{K} to prove the version that is convenient for us.

We will use the following definitions and facts regarding algebraic surfaces and the lines they contain; see \cite{K,R} for further details.
A \emph{quadric} is an algebraic surface of degree two (which may be a union of two planes).
Any three lines are contained in at least one quadric.
A surface $S$ is \emph{singly-ruled} if each point of $S$ is contained in exactly one line that is contained in $S$, 
and $S$ is \emph{doubly-ruled} if each point of $S$ is contained in exactly two lines that are contained in $S$.
A doubly-ruled surface has two \emph{rulings}, i.e., two families of lines such that any two lines within a family are disjoint, but every line from one family intersects every line of the other family.

Given two sets $L,M$ of lines, we write $I(L,M)$ for the number of intersection points between $L$ and $M$, i.e., points contained in at least one line of $L$ and at least one line of $M$.

\begin{lemma}\label{lem:guthkatz}
Let $L,M$ be two finite sets of lines in $\F^3$.
Assume that $|L|\leq |M|$, 
and if $\F$ has positive characteristic $p$,
assume that $|L|=O(p^2)$.
Suppose that no quadric contains $s$ lines of $L$ and $t$ lines of $M$.
Then
\[ I(L,M) = O\left(|L|^{1/2}|M| + t|L| + s|M|\right).\]
\end{lemma}
\begin{proof}
By interpolation (see \cite[Lemma 10]{K}),
there is a surface $S$ of degree $O(|L|^{1/2})$ such that every line of $L$ is contained in $S$.
We decompose $S$ into irreducible components $S_1,\ldots, S_n$.

First we count the points where a line $\ell\in L$ intersects a line $m\in M$,
 with $\ell$ contained in a component $S_i$ that does not contain $m$.
A line $m\in M$ has $O(|L|^{1/2})$ intersection points with components of $S$ that do not contain $m$,
so altogether there are $O(|L|^{1/2}|M|)$ intersection points of this kind.
It remains to count the intersection points between lines that lie in the same component.

We assign every line of $L$ or $M$ that is contained in $S$ to one component containing that line; if a line is contained in more than one component $S_i$, we assign it to the component with smallest index $i$.
Let $L_i$ and $M_i$ be the sets of lines from $L$ and $M$ that are assigned to $S_i$.
Note that $\sum |L_i| = |L|$ and $\sum |M_i| \leq |M|$.
To count the remaining intersection points it suffices to count, for each component $S_i$, the intersection points between a line from $L_i$ and a line from $M_i$.

We count the intersection points that occur in components
of $S$ that are quadrics or planes.
By assumption,
such a component $S_i$ contains either less than $s$ lines from $L$,
or it contains less than $t$ lines from $M$.
In the first case, we count less than $s|M_i|$ intersection points, 
and in the second case we count less than $t|L_i|$ intersection points.
It follows that there are 
less than $s\sum |M_i|+t\sum |L_i| \leq t|L|+s|M|$ intersection points of this kind.

Next we count the intersection points that occur in singly-ruled components of $S$.
Let $S_i$ be such a component.
By \cite[Proposition 55]{K}, there are at most two special lines in $S_i$ that intersect infinitely many other lines in $S_i$.
These contribute at most $2\max\{|L_i|,|M_i|\}$ intersections within $S_i$,
and $O(|M|)$ overall.
Any other line intersects at most $\deg(S_i)$ lines in $S_i$, resulting in at most $\sum \deg(S_i)|L_i|=O(|L|^{1/2}|M|)$ intersection points. 

Finally, 
we count intersection points inside non-ruled components of $S$.
In characteristic zero, we use \cite[Corollary 21]{K},
which states that a non-ruled component $S_i$ contains $O(\deg(S_i)^3)$ intersection points between lines contained in it.
In positive characteristic $p$, the remark before \cite[Corollary 40]{K} claims the same fact if the degree of $S_i$ is less than the characteristic, which corresponds to the condition $|L_i|=O(p^2)$.
Thus there are altogether at most $O(|L|^{3/2})\leq O(|L|^{1/2}|M|)$ intersection points of this kind.
%Combining all of the above, we get the stated bound.
%\[ I(L,M) = O(|L|^{1/2}|M| + k|L|^{1/2} + |L|^{3/2}).\]
\end{proof}

The bound in Lemma \ref{lem:guthkatz} is optimal for $k\geq |L|^{1/2}$ in a similar way to Theorem \ref{thm:rudnev},
since one can take a quadric containing $k-1$ lines of $L$ and all lines of $M$,
to get $(k-1)|M|$ intersection points.
The condition on lines in a quadric could be refined somewhat by considering only planes and doubly-ruled surfaces (for instance, over $\R$ an ellipsoid contains no lines, so does not need to be excluded).

%\newpage

\section{The point-plane incidence bound}

We are now ready to give our alternative proof of Theorem \ref{thm:rudnev}.
In fact, we prove a slightly stronger statement, 
which has a weaker condition on collinearities.
The rationale is that in Theorem \ref{thm:rudnev},
 the condition that no line contains $k$ points of $P$ is somewhat wasteful, since such a line only leads to many incidences if it is also contained in many planes of $Q$.

\begin{theorem}\label{thm:rudnevrefined}
Let $P$ be a finite set of points in $\F^3$
and $Q$ a finite set of planes in $\F^3$.
Assume that $|P|\leq |Q|$, 
and if $\F$ has positive characteristic $p$,
assume that $|P|=O(p^2)$.
Suppose that no line contains $s$ points of $P$ and is contained in $t$ planes of $Q$.
Then
\[ I(P,Q) = O\left(|P|^{1/2}|Q| + t|P|+ s|Q|\right).\]
\end{theorem}
\begin{proof}
We deduce Theorem \ref{thm:rudnev} from Lemma \ref{lem:guthkatz} using the maps $\varphi$ and $\psi$ introduced in Section \ref{sec:map}.
By applying a generic rotation\footnote{This could technically be a problem when $\F$ is a finite field and $P$ and $Q$ are large, 
since there may not be enough rotations to ensure the properties we assume. 
However, we can pass to any infinite extension of $\F$, 
and obtain the stated bound over that extension, which then implies the bound over $\F$.} to $\F^3$, we can assume that no points of $P$ are on the $yz$-plane, and that all planes in $Q$ intersect $\lambda$ (the $z$-axis) and $\pi$ (the plane $x=1$) and do not contain $\lambda$.
We set $L = \varphi(P)$ and $M = \psi(Q)$.
Due to the generic rotation, we can assume that the lines in $L$ are pairwise disjoint, and the lines in $M$ are pairwise disjoint.

By Lemma \ref{lem:map}, a point-plane incidence $p\in q$ corresponds to an intersection point between $\varphi(p)$ and $\psi(q)$.
Since the lines in $L$ are pairwise disjoint and the lines in $M$ are pairwise disjoint,
it follows that $I(P,Q) = I(L,M)$.
The theorem now follows from Lemma \ref{lem:guthkatz}, 
if we show that there is 
no quadric containing $s$ lines of $L$ and $t$ lines of $M$.

Suppose a quadric $S$ contains the lines 
%$\ell_1,\ldots, \ell_1_k$ 
$\varphi(p_1),\ldots,\varphi(p_s)$
of $L$ and the lines 
%$m_1,\ldots, m_k$ 
$\psi(q_1),\ldots,\psi(q_t)$
of $M$.
The lines 
%$\ell_1_i$ 
$\varphi(p_i)$
are disjoint, so they are in the same ruling of $S$.
%$m_j$.
The lines $\psi(q_j)$ must lie in the other ruling of $S$, since there is a line at infinity intersecting all the $\varphi(p_j)$ but none of the $\psi(q_j)$ (by Lemma \ref{lem:map} and the remark right after it).
It follows that every $\varphi(p_i)$ intersects every $\psi(q_j)$,
which means that each of the points $p_1,\ldots, p_s$ is incident with each of the planes $q_1,\ldots, q_t$.
This is only possible if the points $p_1,\ldots, p_s$ lie on a line that is contained in each of the planes $q_1,\ldots, q_t$, which would contradict the assumption of the theorem.
\end{proof}

%%\newpage


\begin{thebibliography}{99}

\bibitem{GK}
L. Guth and N.H. Katz,
\emph{On the Erd{\H o}s distinct distances problem in the plane},
Annals of Mathematics {\bf 181}, 155--190, 2015.

\bibitem{K}
J. Koll\'ar,
\emph{Szemer\'edi--Trotter-type theorems in dimension $3$},
Advances in Mathematics {\bf 271}, 30--61, 2015.

\bibitem{RRS}
O. Roche-Newton, M. Rudnev,  and I.D. Shkredov,
\emph{New sum-product type estimates over finite fields},
Advances in Mathematics \textbf{293},
589--605, 2016.

\bibitem{R}
M. Rudnev,
\emph{On the number of incidences between points and planes in three dimensions},
{\tt arXiv:1407.0426}, 2014.
To appear in Combinatorica.

\bibitem{SZ}
S. Stevens and F. de Zeeuw,
\emph{An improved point-line incidence bound over arbitrary fields},
{\tt arXiv:1609.06284}, 2016.

\end{thebibliography}
\end{document}